\documentclass[11pt,a4paper]{amsart}
\usepackage{amsfonts,amscd,amssymb,amsthm}
\usepackage{amsfonts}
\usepackage{amssymb}
\usepackage{amsmath}
\usepackage{amsxtra}
\usepackage{amscd}
\usepackage{amsthm}
\usepackage{eucal}
\usepackage{array}
\usepackage{graphicx}
\usepackage{tikz}
\usepackage{mathrsfs} 
\usepackage{enumerate}
\usepackage{hyperref}
\usepackage{bm} 
\DeclareMathOperator{\Tor}{Tor}
\DeclareMathOperator{\lcm}{lcm}
\DeclareMathOperator{\indm}{indm}

\def\NZQ{\mathbb}               

\def\ZZ{{\NZQ Z}}

\def\FF{{\NZQ F}}

\def\ab{{\mathbf a}}

\def\xb{{\mathbf x}}

\let\Dirsum=\bigoplus

\newtheorem{theorem}{Theorem}[section]
\newtheorem{lemma}[theorem]{Lemma}
\newtheorem{proposition}[theorem]{Proposition}
\newtheorem{corollary}[theorem]{Corollary}
\newtheorem{conjecture}[theorem]{Conjecture}

\theoremstyle{remark}
\newtheorem{remark}[theorem]{Remark}
\newtheorem{notation}[theorem]{Notation}

\theoremstyle{definition}
\newtheorem{example}[theorem]{Example}

\newtheorem{definition}[theorem]{Definition}
\DeclareMathOperator{\reg}{reg}
\DeclareMathOperator{\aim}{aim}
\DeclareMathOperator{\mat}{mat}

\begin{document}
\title{Squarefree powers of edge ideals of forests}

\author[N. Erey]{Nursel Erey}
\author[T. Hibi]{Takayuki Hibi}
\address{Nursel Erey, Gebze Technical University, Department of Mathematics, 41400 Kocaeli, Turkey}
\email{nurselerey@gtu.edu.tr}

\address{Takayuki Hibi, Department of Pure and Applied Mathematics, Graduate School of Information Science and Technology, Osaka University, Suita, Osaka 565--0871, Japan}
\email{hibi@math.sci.osaka-u.ac.jp}




\begin{abstract}
 Let $I(G)^{[k]}$ denote the $k$th squarefree power of the edge ideal of $G$. When $G$ is a forest, we provide a sharp upper bound for the regularity of $I(G)^{[k]}$ in terms of the $k$-admissable matching number of $G$. For any positive integer $k$, we classify all forests $G$ such that $I(G)^{[k]}$ has linear resolution. We also give a combinatorial formula for the regularity of $I(G)^{[2]}$ for any forest $G$.
\end{abstract}

\maketitle

\section{Introduction}

Let $G$ be a finite simple graph with the vertex set $V(G)=\{x_1,\dots ,x_n\}$ and the edge set $E(G)$. Let $\Bbbk$ be a field and let $S=\Bbbk[x_1,\dots ,x_n]$ be the polynomial ring in $n$ variables over $\Bbbk$. The \textit{edge ideal} of $G$, denoted by $I(G)$, is the monomial ideal generated by $x_ix_j$ such that $\{x_i,x_j\}\in E(G)$. Computation of Castelnuovo-Mumford regularity of edge ideals and their powers is a challenging problem in commutative algebra which has led to extensive literature. 

Matchings in graphs appeared in the context of bounding or computing regularity. For example, it is well known that the regularity $\reg(I(G))$ of edge ideal of $G$ is bounded below by $\indm(G)+1$ \cite{K} and above by $\mat(G)+1$ \cite{HVT} where $\indm(G)$ and $\mat(G)$ denote respectively the induced matching number and the matching number of the graph $G$. It is also known that such lower bound is attained when $G$ is a chordal graph \cite{HVT}. These bounds were generalized to powers of edge ideals in \cite{BBH, BHT}. In particular, for any positive integer $k$, the following inequalities hold:
\[2k+\indm(G)-1 \leq \reg(I(G)^k) \leq 2k+\mat(G)-1. \]
The authors of \cite{BHT} proved that the lower bound is attained when $G$ is a forest and it was conjectured in \cite{BBHsurvey} that such lower bound should be also attained by chordal graphs. 

In this article, we investigate squarefree powers of edge ideals. The $kth$ \textit{squarefree power} $I(G)^{[k]}$ of edge ideal of a graph $G$ is generated by the squarefree monomials in the $k$th ordinary power $I(G)^{k}$. If $k>\mat(G)$, then $I(G)^{[k]}=(0)$. The study of squarefree powers was initiated in \cite{BHZ} and continued in \cite{EHHM}. Our motivation to study such powers is twofold. Firstly, thanks to the Restriction Lemma (Lemma~\ref{restrict}) the regularity of $I(G)^{[k]}$ is bounded above by that of $I(G)^{k}$. This suggests that squarefree powers might be useful in the study of ordinary powers. For instance, if the $k$th squarefree power does not have linear resolution, then the $k$th ordinary power cannot have linear resolution either. The second part of our motivation comes from the fact that the generators of $I(G)^{[k]}$ correspond to the matchings in $G$ of size $k$. This makes a close connection between squarefree powers of edge ideals and the theory of matchings in graphs.

In this article, we introduce the concept of $k$-\textit{admissable matching} of a graph. A matching $M$ is called $k$-admissable if there exists a partition of $M$ that satisfy certain conditions, see Definition~\ref{def: k-admissable matching}. A $1$-admissable matching is the same as an induced matching. Therefore, $k$-admissable matchings can be seen as generalization of induced matchings. The $k$-\textit{admissable matching number} of $G$, denoted by $\aim(G,k)$, is the maximum size of a $k$-admissable matching. Our first main result (Theorem~\ref{thm: upper bound for forest}) gives an upper bound for the regularity of squarefree powers of edge ideals of forests:

\begin{theorem}
	If $G$ is a forest, then $\reg(I(G)^{[k]})\leq \aim(G,k)+k$ for every $1\leq k\leq \mat(G)$.
\end{theorem}

In Theorem~\ref{thm: second power formula} we show that the upper bound above is attained when $k=2$:

\begin{theorem}
		If $G$ is a forest with $\mat(G)\geq 2$, then $\reg(I(G)^{[2]})=\aim(G,2)+2$.
\end{theorem}

Our second main result (Theorem~\ref{thm:characterize linear resolution}) gives a complete classification of forests $G$ for which $I(G)^{[k]}$ has linear resolution:

\begin{theorem}
Let $k$ be a positive integer and let $G$ be a forest with $\mat(G)\geq k$. Then $\reg(I(G)^{[k]})=2k$ if and only if $\aim(G,k)=k$.
\end{theorem}

As a consequence of the above theorem, we show that for any forest $G$ and $1\leq k< \mat(G)$, if $I(G)^{[k]}$ has linear resolution, then $I(G)^{[k+1]}$ has linear resolution as well.

\section{Preliminaries}
\subsection{Definitions and notations}
Let $G$ be a finite simple graph with the vertex set $V(G)$ and the edge set $E(G)$. Given a vertex $x$ in $G$, we say $y$ is a \textit{neighbor} of $x$ if $\{x,y\}\in E(G)$. We denote the set of all neighbors of $x$ by $N_G(x)$. We set $N_G[x]=N_G(x)\cup \{x\}$. We say the \textit{degree} of $x$ in $G$ is $d$ if $x$ has exactly $d$ neighbors. A vertex of degree $0$ is called an \textit{isolated vertex}. A vertex of degree $1$ is called a \textit{leaf}. A \textit{complete graph} on $n$ vertices is denoted by $K_n$.

We say $H$ is a \textit{subgraph} of $G$ if $V(H)\subseteq V(G)$ and $E(H)\subseteq E(G)$. A subgraph $H$ of $G$ is called an \textit{induced subgraph}  if for any two vertices $x,y$ in $H$, $\{x,y\}\in E(H)$  if and only if $\{x,y\}\in E(G)$. For any $U\subseteq V(G)$, the induced subgraph of $G$ on $U$ is the graph with the vertex set $U$ and the edge set $\{\{x,y\} : x,y\in U \text{ and } \{x,y\} \in E(G)\}$. For any $U\subseteq V(G)$, we denote by $G-U$ the induced subgraph of $G$ on $V(G)\setminus U$.

A graph $G$ is called \textit{connected} if any two vertices of $G$ are connected by a path in $G$. A maximal connected subgraph of $G$ is called a \textit{connected component} of $G$. We say $G$ is a \textit{forest} if $G$ has no cycle subgraphs. A connected forest is called a \textit{tree}.

A \textit{matching} of $G$ is a collection of edges which are pairwise disjoint. The \textit{matching number} of $G$, denoted by $\mat(G)$, is defined by
\[\mat(G)=\max\{|M|: M \text{ is a matching of } G\}.\]
A matching $M=\{e_1,\dots ,e_k\}$ of $G$ is called an \textit{induced matching} of $G$ if the induced subgraph of $G$ on $\cup_{i=1}^ke_i$ consists of the edges $e_1,\dots ,e_k$. The \textit{induced matching number} of $G$, denoted by $\indm(G)$, is defined by
\[\indm(G)=\max\{|M|: M \text{ is an induced matching of } G\}.\]
Clearly, $\indm(G)\leq \mat(G)$ for any graph $G$. An induced matching of size $2$ is called a \textit{gap}. If $\{e_1,e_2\}$ is a gap in $G$, we say the edges $e_1$ and $e_2$ \textit{form a gap} in $G$. A matching $M$ of $G$ is called a \textit{perfect matching} if for every vertex $x$ of $G$, there is an edge $e\in M$ such that $x\in e$.

For any positive integer $n$, we denote $\{1,\dots ,n\}$ by $[n]$.

Let $G$ be a graph with the vertex set $V(G)=\{x_1,\dots ,x_n\}$. Let $\Bbbk$ be a field and let $S=\Bbbk[x_1,\dots ,x_n]$ be the polynomial ring in $n$ variables over $\Bbbk$. The \textit{edge ideal} of $G$, denoted by $I(G)$, is the monomial ideal defined by
\[I(G)=(x_ix_j : \{x_i,x_j\} \text{ is an edge of } G).\]

By abuse of notation, we will use an edge $e=\{x_i,x_j\}$ of $G$ interchangeably with the monomial $x_ix_j$. For any $1\leq k\leq \mat(G)$, we define the $k$th \textit{squarefree power} of the edge ideal of $G$ by
\[I(G)^{[k]}=(e_1\dots e_k : \{e_1,\dots ,e_k\} \text{ is a matching of } G).\] 
We set $I(G)^{[k]}=(0)$ when $k>\mat(G)$. For any homogeneous ideal $I\subset S$, the \textit{(Castelnuovo-Mumford) regularity} of $I$ is defined by
\[\reg(I)=\max\{j-i: b_{i,j}(I)\neq 0 \}\]
where $b_{i,j}(I)$ denote the \textit{graded Betti numbers} in the minimal graded free resolution of $I$. An ideal $I$ generated in degree $d$ is said to have a \textit{linear resolution} if $b_{i,i+j}(I)=0$ for all $j\neq d$.

\begin{figure}[hbt!]
	\centering
	\begin{tikzpicture}
	[scale=1.0, vertices/.style={draw, fill=black, circle, minimum size = 4pt, inner sep=0.5pt}, another/.style={draw, fill=black, circle, minimum size = 3.5pt, inner sep=0.1pt}]
	\node[another, label=below:{}] (a) at (0,0) {};
	\node[another, label=below:{}] (b) at (1, 0) {};
	\node[another, label=above:{}] (c) at (-1, 0) {};
	\node[another, label=below:{}] (d) at (1.5,0.75) {};
	\node[another, label=above:{}] (e) at (2,0) {};
	\node[another, label=below:{}] (f) at (1.5,-0.75) {};
	\node[another, label=below:{}] (g) at (-1.5,0.75) {};
	\node[another, label=below:{}] (h) at (-2,0) {};
	\node[another, label=below:{}] (i) at (-1.5,-0.75) {};
	\foreach \to/\from in {a/b, a/c, b/d, b/e, b/f, c/g, c/h,c/i}
	\draw [-] (\to)--(\from);
	\end{tikzpicture}
	\caption{\label{graph mat} A graph $G$ with $\indm(G)=\mat(G)=2$.}
\end{figure}
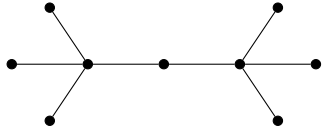

\subsection{Background}
In this section, we collect some results that will be useful to prove our results. The following lemma shows the existence of a certain kind of leaf in forests. 
\begin{lemma}\cite[Proposition~4.1]{JK} \label{lem:jacques}
	Let $T$ be a forest containing a vertex of degree at least two. Then there exists a vertex $v$ with neighbors $v_1,\dots ,v_n$ where $n\geq 2$ and $v_1,\dots ,v_{n-1}$ have degree one.
\end{lemma}

\begin{remark}\cite[Remark~2.6]{BHT}\label{rk: adding variable}
	Let $R=\Bbbk[x_1,\dots ,x_n]$ and let $I\neq R$ be a homogeneous ideal in $R$. Let $y$ be a new indeterminate and let $S=R[y]$. Then $\reg(I)=\reg(I+(y))$.
\end{remark}
\begin{theorem}\cite[Theorem~4.7]{BHT}\label{thm:reg of forest}
	If $G$ is a forest, then $\reg(I(G)^k)=2k+\indm(G)-1$.
\end{theorem}

The authors of \cite{BHZ} proved the surprising result that the highest non-vanishing squarefree power of an edge ideal has linear resolution. This result will be crucial in the proof of Theorem~\ref{thm: upper bound for forest}.
\begin{theorem} \label{thm: highest power has linear quotients}
	{\em(}\cite[Theorem~5.1]{BHZ}{\em)}
	Let $G$ be a graph with matching number at least one. Then $I(G)^{[\mat(G)]}$ has linear quotients and thus it has linear resolution.
\end{theorem}

\begin{lemma}[Restriction Lemma]
	{\em (\cite[Lemma 4.4]{HHZ})}
	\label{restrict}
	Let $I\subset S$ be a monomial ideal, and let $\FF$ be its minimal multigraded free $S$-resolution. Let $G(I)$ denote the minimal set of monomial generators of $I$. Furthermore, let $m$ be a monomial. We set
	\[
	I^{\leq m}=(u\in   G(I)\:\;  u|m ).
	\]
	Let $F_i=\Dirsum_{j}S(-\ab_{ij})$ be the $i$th free module in $\FF$. Then $\FF^{\leq m}$ with $$F_i^{\leq m}= \Dirsum_{j,\;  \xb^{\ab_{ij}}|m}S(-\ab_{ij})$$ is a subcomplex of $\FF$ and the minimal multigraded free resolution of $I^{\leq m}$.
\end{lemma}

We will use the following consequence of Lemma~\ref{restrict}:

\begin{corollary}\cite[Corollary~1.3]{EHHM}
	\label{cor:restriction induced subgraph}
	Let $H$ be an induced subgraph of $G$. Then $b_{i,\ab}(I(H)^{[k]})\leq b_{i,\ab}(I(G)^{[k]})$ for all $i$ and $\ab \in \ZZ^n$. In particular,   $\reg(I(H)^{[k]})\leq  \reg(I(G)^{[k]})$.
\end{corollary}

The following result is well-known, see for example Lemma~3.1 in the survey article \cite{H}.
\begin{lemma}\label{lem: key lemma}
	For any homogeneous ideal $I\subset S$ and any homogeneous element $m\in S$ of degree $d$ the short exact sequence
	\[0 \rightarrow \frac{S}{I:m}(-d) \rightarrow \frac{S}{I} \rightarrow \frac{S}{I+(m)} \rightarrow 0\]
	yields the following regularity bound for $I$:
	\[\reg(I)\leq \max\{\reg(I:m)+d, \reg(I+(m)) \}.\]
\end{lemma}


\section{$k$-admissable matchings}

In this section, we define $k$-admissable matching of a graph, and we make some observations about their properties.
\begin{definition}
	For any positive integers $k$ and $n$, we call a sequence $(a_1,\dots ,a_n)$ of integers a \textit{$k$-admissable sequence} if the following conditions are satisfied:
	\begin{enumerate}
		\item $a_i\geq 1$ for each $i=1,\dots ,n$
		\item $a_1+\dots +a_n\leq n+k-1$.
	\end{enumerate}
	
\end{definition}


\begin{definition}\label{def: k-admissable matching}
	Let $G$ be a graph with matching number $\mat(G)$. Let $M$ be a matching of $G$. For any $1\leq k \leq \mat(G)$ we say $M$ is $k$-\textit{admissable matching} if there exists a sequence $M_1,\dots ,M_r$ of non-empty subsets of $M$ such that  
	\begin{enumerate}
		\item $M=M_1\cup  \dots \cup M_r$,
		\item  $M_i\cap M_j=\emptyset$ for all $i\neq j$,
		\item for all $i\neq j$, if $e_i\in M_i$ and $e_j\in M_j$, then $\{e_i,e_j\}$ is a gap in $G$, 
		\item the sequence $(|M_1|, \dots ,|M_r|)$ is $k$-admissable, and
		\item the induced subgraph of $G$ on $\cup_{e\in M_i}e$ is a forest for all $i\in [r]$.
	\end{enumerate}
In such case, we say $M=M_1\cup \dots \cup M_r$ is a $k$\textit{-admissable partition} of $M$ for $G$.
\end{definition}

\begin{definition}
	The $k$-\textit{admissable matching number} of a graph $G$, denoted by $\aim(G,k)$, is defined by \[\aim(G,k)=\max\{|M| : M \text{ is a } k\text{-admissable matching of } G\}\]
	for $1\leq k \leq \mat(G)$. We define $\aim(G,k)=0$ if $G$ has no $k$-admissable matching.
\end{definition}

\begin{remark}\label{rk:properties of admissable matchings}
	For any graph $G$, one can deduce the following properties of $k$-admissable matchings from the definition.
	\begin{enumerate}
		\item A matching $M$ of $G$ is $1$-admissable if and only if $M$ is an induced matching of $G$. In particular, $\indm(G)=\aim(G,1)$.
		
		\item Let $G$ be a forest. If $M$ is a non-empty matching of $G$, then $M$ is an $|M|$-admissable matching of $G$. Therefore,  $\aim(G,k)\geq k$ for every $1\leq k \leq \mat(G)$.
		
		\item If $1\leq k <\mat(G)$ and $M$ is a $k$-admissable matching, then $M$ is $(k+1)$-admissable matching. In particular,
		\[\indm(G)=\aim(G,1)\leq \aim(G,2) \leq \dots \leq \aim(G,\mat(G))\leq \mat(G).\]
		Moreover, if $G$ is a forest, then $\aim(G,\mat(G))=\mat(G)$.
		
		\item \label{rk:properties item induced subgraph} If $H$ is an induced subgraph of $G$, then $\aim(H,k) \leq \aim(G,k)$ for all $k \in [\mat(H)]$.
		\end{enumerate}
\end{remark}

\begin{lemma}\label{lem:aim(G,k) cannot exceed}
	Let $2\leq k \leq \mat(G)$. If $M$ is $k$-admissable matching, then either $M$ is $(k-1)$-admissable matching, or there exists an edge $e\in M$ such that $M\setminus \{e\}$ is $(k-1)$-admissable matching. Therefore, $\aim(G, k) \leq \aim(G,k-1)+1$.
\end{lemma}
\begin{proof}
	Let $M=M_1\cup \dots \cup M_r$ be a $k$-admissable partition of $M$ for $G$. Then $|M_1|+\dots +|M_r|\leq r+k-1$. If $|M_1|+\dots +|M_r|\leq r+k-2$, then $M$ is a $(k-1)$-admissable matching. Otherwise, $|M_1|+\dots +|M_r|=r+k-1$. Since $k\geq 2$, we must have $|M_i|\geq 2$ for some $i\in[r]$. Let $e\in M_i$. Then $M\setminus \{e\}=M_1\cup \dots \cup M_i\setminus\{e\}\cup \dots \cup M_r$ is a $(k-1)$-admissable partition of $M\setminus \{e\}$ for $G$.
\end{proof}
\begin{example}
	Let $G$ be the graph in Figure~\ref{admissable ex}. Since $G$ has $13$ vertices, $\mat(G)\leq 6$. Let $M_1=\{\{a,b\}, \{c,d\}\}$, $M_2=\{\{f,g\}, \{h,i\}\}$ and $M_3=\{\{j,k\}, \{l,m\}\}$. Then $M=M_1\cup M_2\cup M_3$ is a matching of size $6$ and thus $\mat(G)=6$. In fact, one can show that  $M=M_1\cup M_2\cup M_3$ is a $4$-admissable partition of $M$ for $G$. From Remark~\ref{rk:properties of admissable matchings} it follows that $\aim(G,4)=\aim(G,5)=\aim(G,6)=\mat(G)=6$.
	
	It is not hard to see that the induced matching number of $G$ is $3$. Therefore $\aim(G,1)=3$.
	
	Let $N_1=\{\{a,b\}, \{c,d\}\}$, $N_2=\{\{f,g\}\}$ and $N_3=\{\{j,k\}\}$. Then $N=N_1\cup N_2\cup N_3$ is a $2$-admissable partition of $N$ for $G$. Therefore, $\aim(G,2)\geq 4$. On the other hand, by Lemma~\ref{lem:aim(G,k) cannot exceed} we know that $\aim(G,2)\leq \aim(G,1)+1$. Hence $\aim(G,2)=4$.
	
	Similarly, $U=M_1\cup M_2\cup N_3$ is a $3$-admissable partition of $U$ for $G$. Therefore, $\aim(G,3)\geq 5$. On the other hand, by Lemma~\ref{lem:aim(G,k) cannot exceed} we know that $\aim(G,3)\leq \aim(G,2)+1$. Hence $\aim(G,3)=5$. 
\end{example}

\begin{remark}\label{lem: adding 1 at the end}
	If the sequence $(a_1,\dots ,a_n)$ is $k$-admissable, then so is $(a_1,\dots ,a_n, 1)$.
\end{remark}

\begin{lemma}\label{lem:every nonempty subset is k-admissable}
	Let $M$ be a $k$-admissable matching of a graph $G$. Then any non-empty subset of $M$ is also a $k$-admissable matching of $G$.
\end{lemma}
\begin{proof}
	Let $M=M_1\cup \dots \cup M_r$ be a $k$-admissable partition of $M$ for $G$. Then the sequence $(|M_1|,\dots ,|M_r|)$ is $k$-admissable. Therefore $|M|\leq r+k-1$. Let us assume that $|M|>1$ since otherwise $M$ is the only non-empty subset of itself. It suffices to show that for any $N\subseteq M$ with $|N|=|M|-1$, the matching $N$ is $k$-admissable. Without loss of generality, assume that $N=M\setminus \{e\}$ for some $e\in M_1$. If $M_1=\{e\}$, then $N=M_2\cup \dots \cup M_r$ is a $k$-admissable partition of $N$ for $G$ since $|N|=|M|-1\leq (r-1)+k-1$. Otherwise $N=M_1\setminus \{e\} \cup M_2 \cup \dots \cup M_r$ is a $k$-admissable partition of $N$ for $G$ since $|N|\leq r+k-1$.
\end{proof}
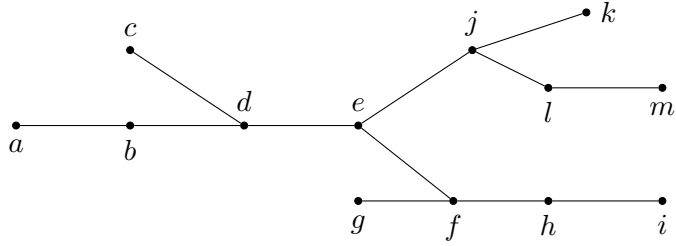
\begin{figure}[hbt!]
	\centering
	\begin{tikzpicture}
	[scale=0.500, vertices/.style={draw, fill=black, circle, minimum size = 4pt, inner sep=0.5pt}, another/.style={draw, fill=black, circle, minimum size = 2.6pt, inner sep=0.1pt}]
	\node[another, label=below:{$a$}] (a) at (-9,0) {};
	\node[another, label=below:{$b$}] (b) at (-6,0) {};
	\node[another, label=above:{$c$}] (c) at (-6, 2) {};
	\node[another, label=above:{$d$}] (d) at (-3,0) {};
	\node[another, label=above:{$e$}] (e) at (0,0) {};
	\node[another, label=below:{$f$}] (f) at (2.5,-2) {};
	\node[another, label=below:{$g$}] (g) at (0,-2) {};
	\node[another, label=below:{$h$}] (h) at (5,-2) {};
	\node[another, label=below:{$i$}] (i) at (8,-2) {};
	\node[another, label=above:{$j$}] (j) at (3,2) {};
	\node[another, label=right:{$k$}] (k) at (6,3) {};
	\node[another, label=below:{$l$}] (l) at (5,1) {};
	\node[another, label=below:{$m$}] (m) at (8,1) {};
	\foreach \to/\from in {a/b, b/d, c/d, d/e, e/f, f/g, f/h, h/i, e/j, j/k, j/l, l/m}
	\draw [-] (\to)--(\from);
	\end{tikzpicture}
	\caption{\label{admissable ex} A graph $G$ with $\aim(G,1)=3,\, \aim(G,2)=4,\, \aim(G,3)=5$ and $\aim(G,4)=\aim(G,5)=\aim(G,6)=\mat(G)=6$.}
\end{figure}


\section{Upper bounds for squarefree powers of edge ideals of forests}
In this section, we provide a sharp upper bound for $\reg(I(G)^{[k]})$ where $G$ is a forest, in terms of $k$-admissable matching number of $G$. A key idea of our method is to work with a special type of vertex in a forest, which we define below.
\begin{definition}
	Let $G$ be a forest with a leaf $x$ and its unique neighbor $y$. We say $x$ is a \textit{distant leaf} if $y$ has at most one neighbor whose degree is greater than $1$. In this case, we say $\{x,y\}$ is a \textit{distant edge}.
\end{definition}
\begin{lemma}\label{lem:distant leaf exists}
	Let $G$ be a forest with at least one edge. Then $G$ has a distant leaf.
\end{lemma}
\begin{proof}
	If $G$ has no vertex of degree at least $2$, then $G$ consists of union of some isolated vertices and $K_2$'s. In such case, every edge is a distant edge. Otherwise, the result follows from Lemma~\ref{lem:jacques}.
\end{proof}

\begin{figure}[hbt!]
	\centering
	\begin{tikzpicture}
	[scale=0.800, vertices/.style={draw, fill=black, circle, minimum size = 4pt, inner sep=0.5pt}, another/.style={draw, fill=black, circle, minimum size = 2.5pt, inner sep=0.1pt}]
	\node[another, label=below:{$x_2$}] (a) at (-0.5,0) {};
	\node[another, label=below:{$x_1$}] (b) at (-3, 0) {};
	\node[another, label=above:{$x_3$}] (c) at (1, 1) {};
	\node[another, label=below:{$x_5$}] (d) at (1,-1) {};
	\node[another, label=above:{$x_4$}] (e) at (3,1) {};
	\node[another, label=below:{$x_6$}] (f) at (3,-1) {};
	\foreach \to/\from in {a/b, a/c, a/d, c/e, d/f}
	\draw [-] (\to)--(\from);
	\end{tikzpicture}
	\caption{\label{graph} A tree $G$ with distant leaves $x_4$ and $x_6$.}
\end{figure}
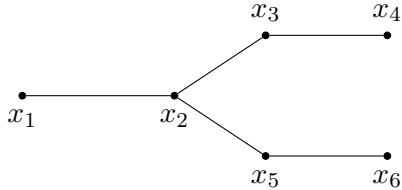

\begin{lemma}\label{lem: matching number leaf edge removal}
	Let $G$ be a forest with a leaf $x$. Then $\mat(G)=\mat(G-\{x,y\})+1$ where $y$ is the unique neighbor of $x$.
\end{lemma}
\begin{proof}
	Any matching of $G-\{x,y\}$ can be extended to a matching of $G$ by adding the edge $\{x,y\}$. Therefore, $\mat(G)\geq \mat(G-\{x,y\})+1$. On the other hand, let $M=\{e_1,\dots , e_{\mat(G)}\}$ be a matching of $G$ of maximum size. If no edge of $M$ contains the vertex $y$, then $M$ is also a matching of $G-\{x,y\}$ and we get $\mat(G)\leq \mat(G-\{x,y\})$ as desired. Otherwise, since $M$ is a matching, there is only one edge $e_i \in M$ such that $y\in e_i$. In such case, $M\setminus \{e_i\}$ is a matching of $G-\{x,y\}$ and $\mat(G)-1\leq \mat(G-\{x,y\})$.
\end{proof}
\begin{lemma}\label{lem:colon by leaf edge}
	Let $G$ be a graph with a leaf $x$. If $y$ is the neighbor of $x$, then for all $k\geq 2$, \[I(G)^{[k]}:(xy)=I(G-\{x,y\})^{[k-1]}.\] 
\end{lemma}
\begin{proof}
	If $\mat(G)<k$, then $\mat(G-\{x,y\})<k-1$ by Lemma~\ref{lem: matching number leaf edge removal}. Then the equality is immediate as both ideals are equal to the zero ideal.
	
	Therefore, let us assume that $2\leq k \leq \mat(G)$. It is clear that $I(G-\{x,y\})^{[k-1]}$ is contained in $I(G)^{[k]}:(xy)$. To see the reverse, let $u$ be a monomial in $I(G)^{[k]}:(xy)$. Then there exists a matching $\{e_1,\dots ,e_k\}$ of $G$ such that $uxy$ is divisible by $e_1\dots e_k$. If $x$ divides $e_1\dots e_k$, then we may assume that $e_1=\{x,y\}$ since $y$ is the only neighbor of $x$. Then $u$ is divisible by $e_2\dots e_k$ and $u\in I(G-\{x,y\})^{[k-1]}$ as $\{e_2,\dots ,e_k\}$ is a matching of $G-\{x,y\}$.
	
	Suppose that $x$ does not divide $e_1\dots e_k$. Then $uy$ is divisible by $e_1\dots e_k$. Since $\{e_1,\dots ,e_k\}$ is a matching, we may assume that $y$ does not divide $e_2\dots e_k$. Hence $u$ is divisible by $e_2\dots e_k$ and the result follows as in the previous case.
\end{proof}

\begin{lemma}\label{lem: aim bound for removing distance leaf}
	Let $G$ be a forest with matching number $\mat(G)\geq 2$. Let $x$ be a distant leaf of $G$ with the neighbor $y$. Then for any $2\leq k \leq  \mat(G)$,
	\[\aim(G-\{x,y\},k-1)+1 \leq \aim(G,k).\]
\end{lemma}

\begin{proof}
	By Lemma~\ref{lem: matching number leaf edge removal} we know that the matching number of $G-\{x,y\}$ is at least $k-1$. Then by Remark~\ref{rk:properties of admissable matchings}, the forest $G-\{x,y\}$ has a $(k-1)$-admissable matching. Let $M$ be a $(k-1)$-admissable matching of $G-\{x,y\}$ of maximum cardinality. We will show that $M'=M\cup \{\{x,y\}\}$ is a $k$-admissable matching of $G$. Let $M=M_1 \cup \dots \cup M_r$ be a $(k-1)$-admissable partition of $M$ for $G-\{x,y\}$.
	
If for every $e\in M$ the edges $e$ and $\{x,y\}$ form a gap in $G$, then $M'=M_1\cup \dots \cup M_r \cup \{\{x,y\}\}$ is a $k$-admissable partition of $M'$ for $G$ by Remark~\ref{lem: adding 1 at the end}. So, suppose that there exists an edge $e\in M$ such that $e$ and $\{x,y\}$ do not form a gap in $G$. Since $x$ is a distant leaf of $G$, it follows that $\{\{x,y\},f\}$ is a gap in $G$ for every $f\in M\setminus \{e\}$. Without loss of generality, suppose that $e\in M_1$. Let $M_1'=M_1\cup \{\{x,y\}\}$. Then $M'=M_1'\cup M_2 \cup \dots \cup M_r$ is a $k$-admissable partition of $M'$ for $G$.
\end{proof}

\begin{remark}
	The lemma above is incorrect for an arbitrary leaf $x$. For example, consider the tree $G$ in Figure~\ref{graph} and let $k=2$. Then $\aim(G-\{x_1, x_2\},1)=\indm(G-\{x_1,x_2\})=2$ but $\aim(G,2)=2$.
\end{remark}


\begin{theorem}\label{thm: upper bound for forest}
	If $G$ is a forest, then $\reg(I(G)^{[k]})\leq \aim(G,k)+k$ for every $1\leq k\leq \mat(G)$.
\end{theorem}

\begin{proof}
	We use induction on $|V(G)|+k$. First note that if $k=1$, then the statement follows from Theorem~\ref{thm:reg of forest} as $\aim(G,1)=\indm(G)$. Also, if $k=\mat(G)$, then by Remark~\ref{rk:properties of admissable matchings} we have $\aim(G,\mat(G))=\mat(G)$ and the result follows from Theorem~\ref{thm: highest power has linear quotients}. Therefore, let us assume that $2\leq k<\mat(G)$.
	
	Then by Lemma~\ref{lem:distant leaf exists}, the forest $G$ has a distant leaf $x_1$ and a unique neighbor $y$. Let $x_1,\dots ,x_r$ be the neighbors of $y$ of degree $1$. We set $I_i=I(G)^{[k]}+(x_1y, \dots ,x_iy)$ and $G_i=G-\{x_1,\dots ,x_i\}$ for each $1\leq i \leq r$. Moreover, we set $G_0=G$ and $I_0=I(G)^{[k]}$. Observe that for each $0\leq i\leq r-1$ 
	\[I_i:(x_{i+1}y)=(I(G_i)^{[k]}:(x_{i+1}y))+(x_1,\dots ,x_i)\]
	and $\{x_{i+1},y\}$ is a distant edge of $G_i$. We claim that 
	\begin{equation}\label{eq: reg bound duplicate leaf lemma}
	\reg(I_i:(x_{i+1}y))\leq \aim(G,k)+k-2 \text{ for all } 0\leq i \leq r-1.
	\end{equation}
Indeed, for each $0\leq i \leq r-1$, since $\mat(G_i)=\mat(G)$ and $\mat(G_i-\{x_{i+1},y\})\geq k-1$, we obtain
\[
\begin{array}{cclr}
	\reg(I_i:(x_{i+1}y)) & = & \reg(I(G_i)^{[k]}:(x_{i+1}y)) & \text{(Remark~\ref{rk: adding variable}) }  \\
& = & \reg(I(G_i-\{x_{i+1},y\})^{[k-1]}) & \text{(Lemma~\ref{lem:colon by leaf edge})} \\
& \leq & \aim(G_i-\{x_{i+1},y\}, k-1)+ k-1 &  \text{(induction assumption)}\\
& \leq &  \aim(G_i,k)+k-2 & \text{(Lemma~\ref{lem: aim bound for removing distance leaf})} \\
& \leq & \aim(G,k)+k-2 & (\text{(\ref{rk:properties item induced subgraph}) of Remark~\ref{rk:properties of admissable matchings}})  
\end{array}
\]
which proves Eq.~\eqref{eq: reg bound duplicate leaf lemma}. We can apply Eq.~\eqref{eq: reg bound duplicate leaf lemma} and Lemma~\ref{lem: key lemma} successively to eliminate $x_1$ and its duplicates as follows.
	\begin{eqnarray*}
		\reg(I(G)^{[k]})&\leq & \max\{\reg(I(G)^{[k]}:(x_1y))+2, \reg(I_1)\} \\
		& \leq & \max\{\aim(G,k)+k, \reg(I_1:(x_2y))+2, \reg(I_2)\}\\
		&\leq & \max\{\aim(G,k)+k,  \reg(I_2:(x_3y))+2, \reg(I_3)\} \\
		&\leq & \vdots\\
		&\leq & \max\{\aim(G,k)+k, \reg(I_r)\}.
	\end{eqnarray*}
	Therefore, it suffices to show that $\reg(I_r)\leq \aim(G,k)+k$. By Lemma~\ref{lem: key lemma} we have
	\[\reg(I_r)\leq \max\{\reg(I_r:(y))+1, \reg(I_r+(y))\}.\]
	Then it suffices to show that the maximum in the above inequality is at most $\aim(G,k)+k$. Note that by Lemma~\ref{lem: matching number leaf edge removal} we have
	\[\mat(G-\{y\})=\mat(G-\{x_1,y\})=\mat(G)-1\geq k.\]
	Remark~\ref{rk: adding variable}, induction assumption and (\ref{rk:properties item induced subgraph}) of Remark~\ref{rk:properties of admissable matchings} imply
	\[\reg(I_r + (y))=\reg(I(G-\{y\})^{[k]})\leq \aim(G-\{y\},k)+k \leq \aim(G,k)+k.\]
	Since $x_1$ is a distant leaf of $G$, either $N_G(y)=\{x_1,\dots ,x_r\}$ or  $N_G(y)=\{x_1,\dots ,x_r,z\}$ for some vertex $z$ of degree greater than $1$. We will consider these cases separately. 
	
	\textbf{Case 1:} Suppose that  $N_G(y)=\{x_1,\dots ,x_r\}$. Then the induced subgraph of $G$ on $N_G[y]$ is a connected component of $G$ and $\aim(G,k)\geq \aim(G_r,k)+1$ by Remark~\ref{lem: adding 1 at the end}. Since $I_r:(y)=I(G_r)^{[k]}+(x_1,\dots ,x_r)$, by Remark~\ref{rk: adding variable} and induction assumption we get
	\[\reg(I_r:(y))= \reg(I(G_r)^{[k]})\leq \aim(G_r, k)+k \leq \aim(G,k)+k-1. \]
	
	\textbf{Case 2:} Suppose that $N_G(y)=\{x_1,\dots ,x_r,z\}$ for some vertex $z$ of degree greater than $1$. Observe that $I_r
	:(y)=zI(G-\{y,z\})^{[k-1]}+I(G-\{y,z\})^{[k]}+(x_1,\dots ,x_r)$. By Lemma~\ref{lem: key lemma}
	\begin{equation}\label{eq: case 2 goal}
	\reg(I_r:(y))\leq \max\{\reg((I_r:(y)):(z))+1, \reg((I_r:(y))+(z))\}.
	\end{equation}
	We will now show that the maximum in \eqref{eq: case 2 goal} is at most $\aim(G,k)+k-1$ which will complete the proof. Observe that by Remark~\ref{rk: adding variable} we have 
	\begin{equation*}\label{eq:case 2 colon}
	\reg((I_r:(y)):(z))= \reg(I(G-\{y,z\})^{[k-1]})=  \reg(I(G-\{x_1, y,z\})^{[k-1]}). 
	\end{equation*}
	 Applying respectively Remark~\ref{rk: adding variable}, Corollary~\ref{cor:restriction induced subgraph}, induction assumption on $G-\{x_1,y\}$ and Lemma~\ref{lem: aim bound for removing distance leaf}, we obtain
		\begin{eqnarray*}
		\reg((I_r:(y)):(z))&\leq & \reg(I(G-\{x_1, y\})^{[k-1]}) \\
		& \leq & \aim(G-\{x_1,y\},k-1)+k-1 \\
		&\leq & \aim(G,k)+k-2.
	\end{eqnarray*}
Observe that Remark~\ref{rk: adding variable} implies $\reg((I_r:(y))+(z))= \reg(I(G-\{y,z\})^{[k]})$. We may assume that the matching number of $G-\{y,z\}$ is at least $k$ since otherwise the proof is immediate. By induction, we have
\begin{equation*}\label{eq:case 2 addition}
	 \reg(I(G-\{y,z\})^{[k]}) \leq \aim(G-\{y,z\},k)+k.
\end{equation*}
Therefore, it remains to show that $\aim(G-\{y,z\},k)+1\leq \aim(G,k)$. Indeed, keeping Remark~\ref{lem: adding 1 at the end} in mind, any $k$-admissable matching of $G-\{y,z\}$ can be extended to a $k$-admissable matching of $G$ by adding the edge $\{x_1,y\}$.
\end{proof}

\subsection{Second squarefree powers}
The goal of this section is to show that the upper bound in Theroem~\ref{thm: upper bound for forest} is attained when $k=2$. To this end, we will show that Betti numbers do not vanish in certain degrees. The following fact was established in the proof of \cite[Theorem~2.1]{EHHM}.
\begin{lemma}\label{lem:oberwolfach}
		If $M$ is a $1$-admissable matching of $G$ which is also a perfect matching, then for all $1\leq k \leq |M|$ \[b_{|M|-k+1,2|M|}(S/I(G)^{[k]})\neq 0.\]
\end{lemma}
We now extend Lemma~\ref{lem:oberwolfach} to $2$-admissable matchings as follows:
\begin{lemma}\label{lem:2 admissable perfect betti number}
	If $M$ is a $2$-admissable perfect matching of $G$, then for all $2\leq k \leq |M|$ \[b_{|M|-k+1,2|M|}(S/I(G)^{[k]})\neq 0.\]
\end{lemma}
\begin{proof}
	 We may assume that $\mat(G)\geq 2$ as the statement is vacuously true otherwise. If $M$ is $1$-admissable matching, then the result follows from Lemma~\ref{lem:oberwolfach}.
	
	So, let us assume that $M$ is not $1$-admissable. Let \[M=\{\{x_1,x_2\}, \{x_3,x_4\},\dots ,\{x_{2r-1},x_{2r}\}\}\] so that $|M|=r$. Since $M$ is not an induced matching of $G$, without loss of generality, we may assume that $\{x_2,x_3\}$ is an edge of $G$.
	
	\textbf{Claim:} $G$ has exactly $r+1$ edges.
	
	\textit{Proof of the claim:} Let $M=M_1\cup \dots \cup M_q$ be a $2$-admissable partition of $M$ for $G$. By condition (3) of Definition~\ref{def: k-admissable matching} we may assume that both $\{x_1,x_2\}$ and $\{x_3,x_4\}$ are in $M_1$. Since the sequence $(|M_1|,\dots ,|M_q|)$ is $2$-admissable, we have $|M_1|+\dots +|M_q|\leq q+1$. On the other hand, since $|M_1|\geq 2$ and $|M_i|\geq 1$ for all $i\geq 2$, we obtain $|M_1|=2$ and $|M_i|=1$ for each $i\geq 2$. The claim then follows from conditions (3) and (5) of Definition~\ref{def: k-admissable matching} together with the fact that $M$ is a perfect matching of $G$.  
	
	Having proved our claim, we can now write
	\[I(G)=(x_1x_2, x_2x_3, x_3x_4, x_5x_6, \dots ,x_{2r-1}x_{2r}).\] 
	 By Lemma~\ref{lem:colon by leaf edge} we set $J:=I(G)^{[k]}:(x_1x_2)=I(G-\{x_1,x_2\})^{[k-1]}$.
Also we set $K=I(G)^{[k]}+(x_1x_2)$. The short exact sequence \[0 \rightarrow \frac{S}{J}(-2) \rightarrow \frac{S}{I(G)^{[k]}} \rightarrow \frac{S}{K} \rightarrow 0\]
yields the long exact sequence 
\[\cdots \rightarrow \Tor_{r-k+2}(S/K,\Bbbk)_{2r}\rightarrow \Tor_{r-k+1}((S/J)(-2),\Bbbk)_{2r}\rightarrow \Tor_{r-k+1}(S/I(G)^{[k]},\Bbbk)_{2r} \rightarrow \cdots \]
Since $M\setminus \{\{x_1,x_2\}\}$ is perfect induced matching of $G-\{x_1,x_2\}$, by Lemma~\ref{lem:oberwolfach} we have $\Tor_{r-k+1}((S/J)(-2),\Bbbk)_{2r}\neq 0$. Therefore, it suffices to show that $\Tor_{r-k+2}(S/K,\Bbbk)_{2r}=0$.   The short exact sequence \[0 \rightarrow \frac{S}{K:(x_2)}(-1) \rightarrow \frac{S}{K} \rightarrow \frac{S}{K+(x_2)} \rightarrow 0\]
yields the long exact sequence 
\[\cdots \rightarrow \Tor_{r-k+2}((S/(K:(x_2)))(-1),\Bbbk)_{2r}\rightarrow \Tor_{r-k+2}(S/K,\Bbbk)_{2r}\rightarrow \Tor_{r-k+2}(S/(K+(x_2)),\Bbbk)_{2r}\rightarrow \cdots \]

We will now use the fact that all non-zero Betti numbers of a squarefree monomial ideal lie in squarefree multidegrees. Observe that the squarefree monomial ideal $K+(x_2)$ lies in a polynomial ring with less than $2r$ variables, more precisely, $K+(x_2)\subseteq \Bbbk[x_2,\dots ,x_{2r}]$. Then $\Tor_{i}((S/(K+(x_2))),\Bbbk)_{2r}=0$ for every $i$. On the other hand,  observe that 
\[K:(x_2)=(I(G)^{[k]}+(x_1x_2)):(x_2)=x_3I(G-\{x_2,x_3\})^{[k-1]}+I(G-\{x_2,x_3\})^{[k]}+(x_1).\]
Then the squarefree monomial ideal $K:(x_2)$ lies in a polynomial ring with $2r-2$ variables because the variables $x_2$ and $x_4$ have disappeared. Therefore $\Tor_{r-k+2}(S/(K:(x_2)),\Bbbk)_{2r-1}=0$ or, equivalently $\Tor_{r-k+2}((S/(K:(x_2)))(-1),\Bbbk)_{2r}=0$.
\end{proof}

%

We now give a formula for the regularity of $I(G)^{[2]}$ when $G$ is a forest.
\begin{theorem}\label{thm: second power formula}
		If $G$ is a forest with $\mat(G)\geq 2$, then $\reg(I(G)^{[2]})=\aim(G,2)+2$.
\end{theorem}
\begin{proof}
	By Theorem~\ref{thm: upper bound for forest} and Corollary~\ref{cor:restriction induced subgraph} it suffices to show that for every $2$-admissable matching $M$ of $G$, the inequality $\reg(I(H)^{[2]})\geq |M|+2$ holds where $H$ is the induced subgraph of $G$ on $\cup_{e\in M}e$. Note that $M$ is a perfect matching and $2$-admissable matching of $H$. Then by Lemma~\ref{lem:2 admissable perfect betti number} we get 
	\[b_{|M|-2,2|M|}(I(H)^{[2]})=b_{|M|-1,2|M|}(S/I(H)^{[2]})\neq 0\]
	and thus $\reg(I(H)^{[2]})\geq |M|+2$ as desired.
\end{proof}
In particular, Theorem~\ref{thm: second power formula} gives a lower bound for the regularity of second squarefree power of edge ideal of any graph.

\begin{corollary}
	If $G$ is a graph with $\mat(G)\geq 2$, then $\reg(I(G)^{[2]})\geq\aim(G,2)+2$.
\end{corollary} 
\begin{proof}
	Follows from Corollary~\ref{cor:restriction induced subgraph}.
\end{proof}

A graph $G$ that satisfies $\indm(G)=\mat(G)$ is called a \textit{Cameron-Walker graph}. Such graphs were studied from a commutative algebra point of view in \cite{HHKO}. The following proposition shows that the upper bound in Theorem~\ref{thm: upper bound for forest} is sharp.
\begin{proposition}\label{prop: upper bound justify}
	If $G$ is a Cameron-Walker forest, then for all $1\leq k \leq \mat(G)$, $\reg(I(G)^{[k]})=\aim(G,k)+k$. 
\end{proposition}

\begin{proof}
	By Remark~\ref{rk:properties of admissable matchings} it follows that $\aim(G,k)=\indm(G)$ for all $1\leq k \leq \mat(G)$.  By Theorem~\ref{thm: upper bound for forest} we only need to show that $\reg(I(G)^{[k]})\geq\indm(G)+k$. Let $M$ be an induced matching of $G$ of maximum cardinality. Let $H$ be the induced subgraph of $G$ on $\cup_{e\in M}e$. Then $M$ is a perfect matching of $H$. The result then follows from Lemma~\ref{lem:oberwolfach} and Corollary~\ref{cor:restriction induced subgraph}.
\end{proof}

Using the structural classification of Cameron-Walker graphs \cite{HHKO}, for any given positive integer $m$, one can construct a Cameron-Walker tree $G$ with $\indm(G)=\mat(G)=m$. Figure~\ref{graph mat} illustrates an example with $m=2$.

Based on the results of this section and Macaulay2 \cite{mac2} computations, we expect that the upper bound in Theorem~\ref{thm: upper bound for forest} would give the exact formula for the regularity of squarefree powers of edge ideals of forests. Thus, we propose the following conjecture.

\begin{conjecture}\label{conjecture}
	If $G$ is a forest, then $\reg(I(G)^{[k]})= \aim(G,k)+k$ for every $1\leq k \leq \mat(G)$.
\end{conjecture}


\section{Characterization of squarefree powers with linear resolutions}
	In this section, we will classify forests $G$ such that $I(G)^{[k]}$ has linear resolution. From Theorem~\ref{thm:reg of forest} it follows that $I(G)^k$ has linear resolution if and only if $\indm(G)=1$ when $G$ is a forest. So, for ordinary powers, such characterization does not depend on $k$, and the class of forests with induced matching number equal to one is rather small. On the other hand, we will see that linearity of resolution of $I(G)^{[k]}$ depends on both the forest $G$ and the integer $k$. 
	
	Let us briefly recall some definitions about simplicial complexes. A \textit{simplicial complex} $\Delta$ on a finite vertex set $V(\Delta)$ is a collection of subsets of $V(\Delta)$ such that if $F\in \Delta$, then every subset of $F$ is also in $\Delta$. Each element of $\Delta$ is called a \textit{face} of $\Delta$. If $F$ is a maximal face of $\Delta$ with respect to inclusion, then we say $F$ is a \textit{facet} of $\Delta$. We write $\Delta=\langle F_1,\dots ,F_r \rangle$ if $F_1,\dots ,F_r$ are all the facets of $\Delta$. We say $\Delta$ is \textit{connected} if for every pair of vertices $u$ and $v$ there exists a sequence $F_1,\dots, F_s$ of facets of $\Delta$ such that $u\in F_1$, $v\in F_s$ and $F_i\cap F_{i+1}\neq \emptyset$ for each $i=1,\dots ,s-1$.
\begin{definition}
	Let $I\subseteq S=\Bbbk[x_1,\dots ,x_n]$ be a monomial ideal and let $\alpha=(\alpha_1,\dots ,\alpha_n)\in \mathbb{N}^n$ be a multidegree. The \textit{upper-Koszul simplicial complex} associated with $I$ at degree $\alpha$, denoted by $K^\alpha (I)$, is the simplicial complex over $V=\{x_1,\dots ,x_n\}$ whose faces are:
	\[\displaystyle \Bigg\{W\subseteq V \, | \, \frac{x_1^{\alpha_1}\dots x_n^{\alpha_n}}{\prod_{u\in W}u} \in I\Bigg\}\]
\end{definition}
Hochster's formula (\cite[Theorem~1.34]{MS}) describe multigraded Betti numbers of a monomial ideal $I$ in terms of reduced homology groups of upper-Koszul simplicial complexes as follows:
\[b_{i,\alpha}(I)=\dim_\Bbbk \tilde{H}_{i-1}(K^\alpha (I); \Bbbk) \quad \text{for } i\geq 0 \text{ and } \alpha\in \mathbb{N}^n.\]

\begin{notation}\label{not:monomials sets}
	Let $m=x_1^{\alpha_1}\dots x_{n}^{\alpha_n}$ be a monomial in $\Bbbk[x_1,\dots ,x_n]$. To ease the notation, the monomial $m$ and the multidegree $(\alpha_1,\dots ,\alpha_n)$ will be used interchangeably. Moreover, if $m=x_{i_1}\dots x_{i_k}$ is squarefree, we will denote the set $\{x_{i_1},\dots ,x_{i_k}\}$ by $m$.
\end{notation}
\begin{lemma}\label{lem:describe facets}
	Let $I$ be a squarefree monomial ideal minimally generated by $m_1,\dots, m_t$. Let $m=\lcm(m_1,\dots ,m_t)$. Then $K^m(I)=\langle m/m_1, \dots ,m/m_t \rangle$.
\end{lemma}
\begin{proof}
	By definition of the upper-Koszul simplicial complex, it is clear that each $m/m_i$ corresponds to a face of $K^m(I)$. Moreover, $m/m_i$ corresponds to a maximal face since $m_i$ is a minimal monomial generator. Lastly, if $u$ is a monomial that corresponds to a face of $K^m(I)$, then $m/u\in I$. Then there is a monomial $v$ such that $m/u=vm_i$ for some $i\in[t]$. This implies that the face $u$ is contained in the facet $m/m_i$.
\end{proof}
The following lemma is well-known in graph theory.
\begin{lemma}\label{aysels lemma1}
	Let $G$ be a graph with connected components $G_1,\dots ,G_r$. Then $G$ has a perfect matching if and only if $G_i$ has a perfect matching for each $i\in [r]$.
\end{lemma}

\begin{lemma}\label{aysels lemma2}
	Let $G$ be a graph which has a perfect matching. Then for any vertex $x$ of $G$, the graph $G-\{x\}$ has no perfect matching.
\end{lemma}
\begin{proof}
	If a graph has perfect matching, then it has even number of vertices.
\end{proof}
\begin{lemma}\label{aysels lemma3}
	Let $G$ be a graph with connected components $G_1,\dots ,G_r$ where $r\geq 2$. Suppose that $G$ has a perfect matching. If $x\in V(G_1)$ and $y\in V(G_2)$, then $G-\{x,y\}$ has no perfect matching.
\end{lemma}
\begin{proof}
	By Lemma~\ref{aysels lemma1} each $G_i$ has a perfect matching. Let $U_1,\dots, U_t$ be the connected components of $G_1-\{x\}$ and $V_1,\dots, V_s$ be the connected components of $G_2-\{y\}$. Then the connected components of $G-\{x,y\}$ are $U_1,\dots ,U_t,V_1,\dots ,V_s, G_3, \dots ,G_r$. By Lemma~\ref{aysels lemma2} the graph $G_1-\{x\}$ has no perfect matching. Then by Lemma~\ref{aysels lemma1} there exists $U_j$ which has no perfect matching. Since $U_j$ is also a connected component of $G-\{x,y\}$, it follows that $G-\{x,y\}$ has no perfect matching.
\end{proof}
\begin{notation}\label{not: matching product}
	If $M=\{e_1,\dots ,e_k\}$ is a matching of $G$, then we will write $u_M$ for the squarefree monomial $\displaystyle e_1\dots e_k=\prod_{\substack{x_i\in e, \\ e\in M}}x_i$.
\end{notation}

\begin{lemma}\label{lem: connected components k admissable}
	Let $G$ be a graph with a $k$-admissable perfect matching $M$. Let $M=M_1\cup \dots \cup M_r$ be a $k$-admissable partition of $M$ for $G$. Using Notation~\ref{not: matching product} let $x | u_{M_i}$ and $y | u_{M_j}$ for some vertices $x$ and $y$ with $i\neq j$. Then $x$ and $y$ are in different connected components of $G$. 
\end{lemma}
\begin{proof}
	If $\{a,b\}$ is an edge of $G$, then since $M$ is a perfect matching, $a|u_{M_p}$ and $b|u_{M_q}$ for some $p$ and $q$. Since $M$ is $k$-admissable, we get $p=q$. Therefore there is no path in $G$ that connects $x$ and $y$.
\end{proof}

We use Notation~\ref{not: matching product} again to state the next lemma:

\begin{lemma}\label{lem: disconnected simplicial complex}
	Let $H$ be a graph with a $k$-admissable perfect matching $M$ of cardinality $k+1$. Then the simplicial complex $K^{u_M}(I(H)^{[k]})$ is disconnected.
\end{lemma}

\begin{proof}
	Let $V(H)=\{x_1,\dots ,x_{2k+2}\}$. Then $u_M=x_1\dots x_{2k+2}$ and $u_M$ is the least common multiple of minimal monomial generators of $I(H)^{[k]}$. By Lemma~\ref{lem:describe facets}, observe that every facet of $K^{u_M}(I(H)^{[k]})$ consists of $2$ vertices. In fact, $\{x_i,x_j\}$ is a facet of $K^{u_M}(I(H)^{[k]})$ if and only if $H-\{x_i,x_j\}$ has a perfect matching. To see this, let $\mathcal{F}$ be the set of facets of $K^{u_M}(I(H)^{[k]})$. Then by Lemma~\ref{lem:describe facets}
	\begin{align*}
	\{x_i,x_j\}\in \mathcal{F}&\Longleftrightarrow x_ix_j=\frac{u_M}{u_N} \text{ for some matching } N \text{ of } H \text{ of size } k\\
	&\Longleftrightarrow x_ix_j=\frac{u_M}{u_N} \text{ for some matching } N \text{ of } H-\{x_i,x_j\} \text{ of size } k\\
	& \Longleftrightarrow u_N=V(H)\setminus \{x_i,x_j\} \text{ for some matching } N \text{ of } H-\{x_i,x_j\} \text{ of size } k\\
		& \Longleftrightarrow  H-\{x_i,x_j\} \text{ has a perfect matching}.
	\end{align*}
	Let $M=M_1 \cup M_2 \cup \dots \cup M_r$ be a $k$-admissable partition of $M$ for $H$. Then by definition of $k$-admissable sequence, we must have $r\geq 2$. Let $e_1\in M_1$ and $e_2\in M_2$. Then both $e_1$ and $e_2$ are facets of $K^{u_M}(I(H)^{[k]})$. We claim that there is no sequence of faces that connects a vertex of $e_1$ to a vertex of $e_2$. To this end, we will show that if $\{x_i,x_j\}$ is a facet of $K^{u_M}(I(H)^{[k]})$, then $x_ix_j | u_{M_q}$ for some $q\in [r]$. Assume for a contradiction there is a facet $\{x_i,x_j\}$ such that $x_i | u_{M_{i'}}$ and $x_j |u_{M_{j'}}$ for some $i'\neq j'$. Then by Lemma~\ref{lem: connected components k admissable} the vertices $x_i$ and $x_j$ belong to different connected components of $H$. Then by Lemma~\ref{aysels lemma3}, $H-\{x_i,x_j\}$ has no perfect matching, which is a contradiction.
\end{proof}

The authors of \cite{EHHM} classified all forests $G$ such that $I(G)^{[2]}$ has linear resolution, see \cite[Theorem~5.3]{EHHM}. Our next theorem solves this classification problem for any squarefree power.
\begin{theorem}\label{thm:characterize linear resolution}
	Let $k\geq 1$ be an integer and let $G$ be a forest with $\mat(G)\geq k$. Then $\reg(I(G)^{[k]})=2k$ if and only if $\aim(G,k)=k$.
\end{theorem}
\begin{proof}
	If $\aim(G,k)=k$, then by Theorem~\ref{thm: upper bound for forest} it follows that $\reg(I(G)^{[k]})=2k$. Suppose that $\aim(G,k)\neq k$. Then $\aim(G,k)>k$ by Remark~\ref{rk:properties of admissable matchings}. Let $N$ be a $k$-admissable matching of $G$ of cardinality $\aim(G,k).$ Then by Lemma~\ref{lem:every nonempty subset is k-admissable} there exists a $k$-admissable matching $M$ of $G$ which has $k+1$ elements. Let $H$ be the induced subgraph of $G$ on $\cup_{e\in M}e$. Then $M$ is a perfect matching of $H$. By Corollary~\ref{cor:restriction induced subgraph} we get $\reg(I(G)^{[k]})\geq \reg(I(H)^{[k]})$. By Lemma~\ref{lem: disconnected simplicial complex}, the simplicial complex $K^{u_M}(I(H)^{[k]})$ is disconnected. Then $\dim_\Bbbk\tilde{H}_0(K^{u_M}(I(H)^{[k]}); \Bbbk)>0$. From the Hochster's formula, we get $b_{1,2k+2}(I(H)^{[k]})\neq 0$ and thus $\reg(I(H)^{[k]})\geq 2k+1$.
\end{proof}
Herzog, Hibi and Zheng \cite{HHZ2} proved that if an edge ideal $I(G)$ has linear resolution, then $I(G)^k$ has linear resolution for all $k\geq 1$. It is an open problem to determine for a given integer $k$, whether linearity of resolution of $I(G)^k$ implies the same property for $I(G)^{k+1}$. Relevantly, Theorem~\ref{thm:characterize linear resolution} has an interesting consequence regarding linear resolutions of consecutive squarefree powers:
\begin{corollary}
	Let $G$ be a forest and $1\leq k < \mat(G)$. If $I(G)^{[k]}$ has linear resolution, then $I(G)^{[k+1]}$ has linear resolution.
\end{corollary}
\begin{proof}
	Suppose that $\reg(I(G)^{[k]})=2k$. Then by Theorem~\ref{thm:characterize linear resolution} we get $\aim(G,k)=k$. Now, observe that
	\[
	\begin{array}{cclr}
	\reg(I(G)^{[k+1]}) & \leq & \aim(G,k+1)+k+1 & \text{(Theorem~\ref{thm: upper bound for forest}) }  \\
	& \leq & \aim(G,k)+1+k+1 & \text{(Lemma~\ref{lem:aim(G,k) cannot exceed})} \\
	& = & 2k+2 &  
	\end{array}
	\]
	and thus $I(G)^{[k+1]}$ has linear resolution.
\end{proof}

	\section*{Acknowledgment}
	We thank the anonymous referee for her/his careful reading of the paper and helpful comments.
	

\end{document}